\documentclass[11pt,reqno]{amsart}
\usepackage{amsmath}
\usepackage{amsthm}
\usepackage{amssymb}
\usepackage{enumerate, paralist}

\usepackage{hyperref}

\newdimen\theight
\def\TeXref#1{%
             \leavevmode\vadjust{\setbox0=\hbox{{\tt
                     \quad\quad  {\small \textrm #1}}}%
             \theight=\ht0
             \advance\theight by \lineskip
             \kern -\theight \vbox to
             \theight{\rightline{\rlap{\box0}}%
             \vss}%
             }}%

\usepackage{times}

\newcommand{\R}{\mathbf{R}}
\newcommand{\Q}{\mathbf{Q}}
\newcommand{\N}{\mathbf{N}}

\theoremstyle{plain}

\newtheorem{theorem}{Theorem}[section]
\newtheorem{lemma}[theorem]{Lemma}

\newtheorem{prop}[theorem]{Proposition}

\theoremstyle{definition}

\theoremstyle{remark}

\newtheorem{rem}{Remark}

\newtheorem{claim}{Claim}

\newcommand{\MM}{\mathcal{M}}
\newcommand{\NN}{\mathcal{N}}

\begin{document}

\title{Non-reduction of relations in the Gromov space to Polish actions}

\author[J.A. \'Alvarez L\'opez]{Jes\'us A. \'Alvarez L\'opez}
\address{Departamento de Xeometr\'{\i}a e Topolox\'{\i}a\\
         Facultade de Matem\'aticas\\
         Universidade de Santiago de Compostela\\
         15782 Santiago de Compostela\\ Spain}
\email{jesus.alvarez@usc.es}

\author[A. Candel]{Alberto Candel} 
\address{Department of Mathematics\\ 
	California State University at Northridge\\ 
	18111 Nordhoff Street\\
	Northridge, CA 91330\\ U.S.A.}
\email{alberto.candel@csun.edu}

\thanks{Research of authors supported by Spanish Ministerio de Ciencia e Innovaci\'on, grants MTM2011-25656 and MTM2008-02640}
\date{}

\maketitle

\begin{abstract}
  It is shown that, in the Gromov space of isometry classes of pointed
  proper metric spaces, the equivalence relations defined by existence
  of coarse quasi-isometries or being at finite Gromov-Hausdorff
  distance, cannot be reduced to the equivalence relation defined by
  any Polish action.
\end{abstract}

\tableofcontents

\section{Introduction}\label{s: intro}

Gromov~\cite[Chapter~3]{Gromov1999},~\cite{Gromov1981} described a
space, denoted here by \(\MM_*\), whose points are isometry classes of
pointed complete proper metric spaces. It is endowed with a topology
which resembles the Tychonov topology of \(\R^\N\), or the compact
open topology on the space of continuous functions \(C(\R)\). It also
supports several equivalence relations of geometric interest, like the
relation of being at finite Gromov-Hausdorff distance, $E_{GH}$, and
the relation of being (coarsely) quasi-isometric, $E_{QI}$.

The following concepts relate the complexity of two equivalence
relations on topological spaces, \(E\) over \(X\) and \(F\) over
\(Y\). A map \(\theta : X\to Y\) is called \emph{\((E,F)\)-invariant}
if \(x E x' \Longrightarrow \theta(x) F \theta(x')\) ($\theta$ induces
a mapping $\bar\theta:X/E\to Y/F$). It is said that \(E\) is
\emph{Borel reducible} to \(F\), denoted by \(E\leq_BF\), if there is
an \((E,F)\)-invariant Borel mapping \(\theta : X\to Y\) such that
\(xEx'\Leftrightarrow\theta(x)F\theta(y)\) ($\bar\theta:X/E\to Y/F$ is
injective).  If $E\le_B F$ and $F\le_BE$, then $E$ is said to be
\emph{Borel bi-reducible} with $F$, and is denoted by $E\sim_BF$.  If
the map $\theta$ can be chosen to be continuous, then the terms
``\emph{continuously reducible}'' and ``\emph{continuously
  bi-reducible}'' are used, with notation ``$\le_c$'' and
``$\sim_c$''.

For an example of an equivalence relation, let $G$ be a Polish group acting continuously on a Polish space $X$ (a Polish action). We then let $E^X_G$ denote the orbit equivalence relation whose equivalence classes are exactly the $G$-orbits. For instance,
Hjorth's theory of turbulence~\cite{Hjorth2000},~\cite{Hjorth2002} is
valid for relations defined by Polish actions. The following is our
main result.

\begin{theorem}\label{t: E_GH and E_QI}
  \(E_{GH}\nleq_B E_G^X\) and \(E_{QI}\nleq_B E_G^X\) for any Polish
  group \(G\) and any Polish \(G\)-space \(X\).
\end{theorem}

The theory of turbulence is extended in \cite{AlvarezCandel:turbulent}
to more general equivalence relations on Polish spaces, and it is
applied to \(E_{QI}\) and \(E_{GH}\). This is a non-trivial extension
by Theorem~\ref{t: E_GH and E_QI}.

The proof of Theorem~\ref{t: E_GH and E_QI} uses the following. Let
$E_1$ be the equivalence relation on $\R^\N$ consisting of the pairs
$(x,y)$, with $x=(x_n)$ and $y=(y_n)$, such that there is some
$N\in\N$ so that $x_n=y_n$ for all $n\ge N$ (the relation of eventual
agreement). We have $E_1\nleq_BE^X_G$ for any Polish group $G$ and any
Polish $G$-space $X$ \cite[Theorem 4.2]{KechrisLouveau1997} (see also
\cite[Theorem~8.2]{Hjorth2000} for a different proof).

On the other hand, let $E_{K_\sigma}$ be the equivalence relation on
$\prod_{n=2}^\infty\{1,\dots,n\}$ consisting of the pairs $(x,y)$,
with $x=(x_n)$ and $y=(y_n)$, such that $\sup_n|x_n-y_n|<\infty$. We
have $E\le_BE_{K_\sigma}$ for any $K_\sigma$ equivalence
relation\footnote{Recall that a subset of a topological space is
  called $K_\sigma$ when it is a countable union of compact subsets.}
$E$ \cite[Theorem~17 and Proposition~19]{Rosendal2005}, and therefore
$E_1\le_BE_{K_\sigma}$ because $E_1$ is $K_\sigma$
\cite{Rosendal2005}, \cite[Exercise~8.4.3]{Gao2009}; in particular,
$E_{K_\sigma}\nleq_BE^X_G$ for any Polish group $G$ and every Polish
$G$-space $X$. Therefore Theorem~\ref{t: E_GH and E_QI} follows by
showing that $E_{K_\sigma}\leq_BE_{GH}$ and $E_{K_\sigma}\leq_BE_{QI}$
(Proposition~\ref{p: E_GH and E_QI}).

The relations $E_{GH}$ and $E_{QI}$ resemble the equivalence relation
$E_{\ell_\infty}$ on $\R^\N$ defined by the action of\footnote{Recall
  that $\ell_\infty\subset\R^\N$ is the linear subspace of bounded
  sequences, and $C_b(\R)\subset C(\R)$ is the linear subspace of
  bounded continuous functions.} $\ell_\infty$ on $\R^\N$ by
translations, or the equivalence relation $E_\infty$ on $C(\R)$
defined by the action of $C_b(\R)$. Thus Proposition~\ref{p: E_GH and
  E_QI} has some analogy with the property
$E_{K_\sigma}\sim_BE_{\ell_\infty}$
\cite[Proposition~19]{Rosendal2005}; in particular, $E_1\le_B
E_{\ell_\infty}$ (see also \cite[Theorem~8.4.2]{Gao2009}). It also has
some similarity with the property $E_{K_\sigma}\le_BE_\infty$, which
follows because $E_{\ell_\infty}\le_cE_\infty$; this reduction can be
realized by the map $\R^\N\to C(\R)$, assigning to each element its
canonical continuous piecewise affine extension that is constant on
$(-\infty,0]$.


\section{The Gromov space}\label{ss: Gromov sp}

Let $M$ be a metric space and let $d_M$, or simply $d$, be its
distance function. The \emph{Hausdorff distance} between two non-empty
subsets, $A,B\subset M$, is given by
  \[ 
  H_d(A,B)=\operatorname{max}\biggl\{\sup_{a\in A}\inf_{b\in B}d(a,b),
  \sup_{b\in B}\inf_{a\in A}d(a,b)\biggr\}\;.
  \] Observe that $H_d(A,B)=H_d(\overline{A},\overline{B})$, and
$H_d(A,B)=0$ if and only if $\overline{A}=\overline{B}$. Also, it is
well known and easy to prove that $H_d$ satisfies the triangle
inequality, and its restriction to the family of non-empty compact subsets of
$M$ is finite valued, and moreover complete if $M$ is complete.

Let $M$ and $N$ be arbitrary non-empty metric spaces. A metric on
$M\sqcup N$ is called \emph{admissible} if its restrictions to $M$ and
$N$ are $d_M$ and $d_N$, where $M$ and $N$ are identified with their
canonical injections in $M\sqcup N$. The \emph{Gromov-Hausdorff
  distance} (or \emph{GH distance}) between $M$ and $N$ is defined by
  \[ 
    d_{GH}(M,N)=\inf_dH_d(M,N)\;,
  \] where the infimum is taken over all admissible metrics $d$ on
$M\sqcup N$.  It is well known that
$d_{GH}(M,N)=d_{GH}(\overline{M},\overline{N})$, where $\overline{M}$
and $\overline{N}$ denote the completions of $M$ and $N$,
$d_{GH}(M,N)=0$ if $\overline{M}$ and $\overline{N}$ are
isometric, $d_{GH}$ satisfies the triangle inequality, and
$d_{GH}(M,N)<\infty$ if $\overline{M}$ and $\overline{N}$ are compact.

There is also a pointed version of $d_{GH}$ which satisfies analogous
properties: the (\emph{pointed}) \emph{Gromov-Hausdorff distance}
(or \emph{GH distance}) between two pointed metric spaces, $(M,x)$
and $(N,y)$, is defined by
  \begin{equation}\label{d_GH(M,x;N,y)} d_{GH}(M,x;N,y) =
\inf_d\max\{d(x,y),H_d(M,N)\}\;,
\end{equation} where the infimun is taken over all admissible
metrics  \(d\) on
$M\sqcup N$.

A metric space, or its distance function, is called \emph{proper} (or
\emph{Heine-Borel}) if every open ball has compact closure. This
condition is equivalent to the compactness of the closed balls, which
means that the distance function to a fixed point is a proper
function. Any proper metric space is complete and locally compact, and
its cardinality is not greater than the cardinality of the
continuum. Therefore it may be assumed that their underlying sets are
subsets of \(\R\). With this assumption, it makes sense to consider
the set $\MM_*$ of isometry classes, $[M,x]$, of pointed proper metric
spaces, $(M,x)$. The set $\MM_*$ is endowed with a topology introduced
by M.~Gromov~\cite[Section~6]{Gromov1999},~\cite{Gromov1981}, which
can be described as follows.

For a metric space $X$, two subspaces, $M,N\subset X$, two points,
$x\in M$ and $y\in N$, and a real number $R>0$, let \(
H_{d_X,R}(M,x;N,y)\) be given by
  \[ H_{d_X,R}(M,x;N,y)=\max\biggl\{\sup_{u\in
B_M(x,R)}d_X(u,N),\sup_{v\in B_N(y,R)}d_X(v,M)\biggr\}\;.
  \] Then, for $R,r>0$, let $U_{R,r}\subset\MM_*^2$ denote
the subset of pairs \(([M,x],[N,y])\) for which there is an admissible
metric, $d$, on $M\sqcup N$ so that 
  \[ 
    \max\{d(x,y),H_{d,R}(M,x;N,y)\}<r\;.
  \] 
Let \(\Delta\subset\MM_*^2\) denote the diagonal.

\begin{lemma}\label{l: U_R,r}
  The following properties hold:
    \begin{enumerate}[{\rm(}i{\rm)}]
  
      \item\label{i: bigcap_R,r>0 U_R,r = Delta} $\bigcap_{R,r>0}U_{R,r}=\Delta$;
      
      \item\label{i: U_R,r is symmetric} each $U_{R,r}$ is symmetric;
      
      \item\label{i: if R le S then U_R,r supset U_S,r} if $R\le S$, then $U_{R,r}\supset U_{S,r}$ for all $r>0$;
    
      \item\label{i: U_R,r = bigcup_s<r U_R,s} $U_{R,r}=\bigcup_{s<r}U_{R,s}$ for all $R,r>0$; and
    
      \item\label{i: U_S,r circ U_S,s subset U_R,r+s} \(U_{S,r}\circ U_{S,s}\subset U_{R,r+s}\), where \(S=R+2\max\{r,s\}\).
  
    \end{enumerate}
\end{lemma}

\begin{proof} Items~(\ref{i: bigcap_R,r>0 U_R,r = Delta})--(\ref{i:
    U_R,r = bigcup_s<r U_R,s}) are elementary. To prove~(\ref{i: U_S,r
    circ U_S,s subset U_R,r+s}), let \([M,x],[N,y]\in\MM_*\) and
  \([P,z]\in U_{S,r}(N,y)\cap U_{S,s}(M,x)\). Then there are
  admissible metrics, $d$ on $M\sqcup P$ and $\bar d$ on $N\sqcup P$,
  such that $d(x,z)<r$, $r_0:=H_{d,S}(M,x;P,z)<r$, $\bar d(y,z)<s$ and
  $s_0:=H_{\bar d,S}(N,y;P,z)<s$. Let $\hat d$ be the admissible
  metric on $M\sqcup N$ such that
    \[ 
      \hat d(u,v)=\inf\bigl\{\,d(u,w)+\bar d(w,v)\mid w\in P\,\bigr\}
    \] 
  for all $u\in M$ and $v\in N$. Then
    \[ 
      \hat d(x,y)\le d(x,z)+\bar d(z,y)<r+s\;.
    \] 
  For each $u\in B_M(x,R)$, there is some $w\in P$ such that
$d(u,w)<r_0$. Then
    \[ 
      d_P(z,w)\le d(z,x)+d_M(x,u)+d(u,w)<r+R+r_0<S\;.
    \] 
  So there is some $v\in N$ such that $\bar d(w,v)<s_0$, and we
have
    \[ \hat d(u,v)\le d(u,w)+\bar d(w,v)<r_0+s_0\;.
    \] Hence $\hat d(u,N)<r_0+s_0$ for all $u\in B_M(x,R)$. Similarly,
$\hat d(v,M)<r_0+s_0$ for all $v\in B_N(y,R)$. Therefore $H_{\hat
d,R}(M,x;N,y)\le r_0+s_0<r+s$. Then $[N,y]\in U_{R,r+s}(M,x)$.
\end{proof}

By Lemma~\ref{l: U_R,r}, the sets $U_{R,r}$ form a
base of entourages of a metrizable uniformity on $\MM_*$. Endowed with
the induced topology, $\MM_*$ is what is called the \emph{Gromov
space} in this paper. It is well known that $\MM_*$ is a Polish space 
(see \emph{e.g.}\ Gromov~\cite{Gromov1999} or Petersen~\cite{Petersen1998});
in particular, a countable dense subset is defined by the pointed
finite metric spaces with $\Q$-valued metrics.

\section{Equivalence relations on the Gromov space}\label{ss: equivalence relations}

Recall the following terminology. A map between metric spaces, $\phi:M\to N$, is called
\emph{bi-Lipschitz} if there is some $\lambda\ge1$ such that
 \[
 \lambda^{-1}\,d_M(u,v)\le d_N(\phi(u),\phi(v))\le \lambda\,d_M(u,v)
 \]
 for all $u,v\in M$. The term \emph{$\lambda$-bi-Lipschitz} may be
 also used in this case. A subset $A\subset M$ is called a \emph{net}\footnote{This term is used by Gromov with this meaning \cite[Definition~2.14]{Gromov1999}. Other authors use it with other meanings.} (respectively, \emph{separated}) if there is some $C\ge0$ such that $d_M(x,A)\le C$ for all $x\in M$ (respectively, there is some $\delta>0$ so that $d_M(x,y)\ge\delta$ if $x\ne y$). 
The term \emph{$C$-net} (respectively, \emph{$\delta$-separated}) may be also used in this case. There always exist separated nets \cite[Lemma~9.4]{AlvarezCandel:turbulent}. A 
(\emph{coarse}) \emph{quasi-isometry} of $M$ to
 $N$ is a bi-Lipschitz bijection $\phi:A\to B$ for some nets $A\subset M$
 and $B\subset N$. The existence of a quasi-isometry of $M$ to $N$ is
 equivalent to the existence of a finite sequence of metric spaces,
 $M=M_0,\dots,M_{2k}=N$, such that $d_{GH}(M_{2i-2},M_{2i-1})<\infty$
 and there is a bi-Lipschitz bijection $M_{2i-1}\to M_{2i}$ for all
 $i\in\{1,\dots,k\}$. A \emph{pointed} (\emph{coarse})
 \emph{quasi-isometry} is defined in the same way, by using a pointed
 bi-Lipschitz bijection between nets that contain the distinguished
 points. The existence of a pointed quasi-isometry has an analogous
 characterization involving pointed Gromov-Hausdorff distances and
 pointed bi-Lipschitz bijections.

The following equivalence relations are considered on $\MM_*$:
\begin{itemize}
    
\item The \emph{canonical relation}, $E_{\text{\rm can}}$, is defined
  by varying the distinguished point; \emph{i.e.}, $E_{\text{\rm
      can}}$ consists of the pairs of the form $([M,x],[M,y])$ for any
  proper metric space $M$ and all $x,y\in M$.
    
\item The \emph{Gromov-Hausdorff relation}, $E_{GH}$, consists of the
  pairs $([M,x],[N,y])\in\MM_*^2$ such that
  $d_{GH}(M;N)<\infty$, or, equivalently, $d_{GH}(M,x;N,y)<\infty$.
    
\item The \emph{Lipschitz relation}, $E_{\text{\rm Lip}}$, consists of
  the pairs $([M,x],[N,y])\in\MM_*^2$ such that there is a
  bi-Lipschitz bijection $M\to N$. If $M$ and $N$ are separated, 
  this is equivalent to the existence of a
  pointed bi-Lipschitz bijection $(M,x)\to(N,y)$.
        
\item The \emph{quasi-isometric relation}, $E_{QI}$, consists of the
  pairs $([M,x],[N,y])\in\MM_*^2$ such that there is a
  quasi-isometry of $M$ to $N$, or, equivalently, there is a pointed
  quasi-isometry of $(M,x)$ to $(N,y)$. By the above observations,
  $E_{QI}$ is the smallest equivalence relation over $\MM_*$ that
  contains $E_{GH}\cup E_{\text{\rm Lip}}$.
        
\end{itemize}
Since $E_{\text{\rm can}}\subset E_{GH}\cap E_{QI}$, it follows that
$\MM_*/E_{GH}$ can be identified with the set of classes of proper
metric spaces modulo finite GH distance, and $\MM_*/E_{QI}$ can be
identified with the set of quasi-isometry types of proper metric spaces.

\section{Non-reduction to Polish actions}\label{s: non-reduction}

As indicated in Section~\ref{s: intro}, Theorem~\ref{t: E_GH and E_QI}
follows from the following.

\begin{prop}\label{p: E_GH and E_QI}
	$E_{K_\sigma}\le_cE_{GH}$ and $E_{K_\sigma}\le_cE_{QI}$.
\end{prop}

\begin{proof}
  Let us proof first that $E_{K_\sigma}\le_cE_{QI}$, which is more
  difficult. Consider the metric $d$ on $\R^2$ defined by
  \[
  d((u,v),(u',v'))=
  	\begin{cases}
		|v|+|u-u'|+|v'| & \text{if $u\ne u'$}\\
		|v-v'| & \text{if $u=u'$}\;.
	\end{cases}
  \]
  This is the metric of an $\R$-tree. For each
  $x=(x_n)\in\prod_{n=2}^\infty\{1,\dots,n\}$ and $n\ge2$, let
  \[
  P^\pm_{x,n}=(\textstyle{\sum_{i=2}^ne^{i^2}},\pm
  e^{x_n})\in\R^2\;,\quad M_{x,n}=\{P^+_{x,n},P^-_{x,n}\}\;,
  \]
  and let $M_x:=\bigcup_{n=2}^\infty M_{x,n}$, equipped with the
  restriction $d_x$ of $d$. Given any
  $x=(x_n)\in\prod_{n=2}^\infty\{1,\dots,n\}$, if $A\subset M_x$ is
  $C$-net for some $C\ge0$, it easily follows that
  \begin{gather}
    e^{n^2}\ge C\Longrightarrow A\cap M_{x,n}\ne\emptyset\;,
    \label{e^n^2 ge C}\\
    (e^{n^2}\ge C\quad\&\quad e^{x_n}>C/2)\Longrightarrow
    M_{x,n}\subset A\;.
    \label{e^n^2 ge C & e^{x_n}>C/2}
  \end{gather}
  Let $\theta:\prod_{n=2}^\infty\{1,\dots,n\}\to\MM_*$ be defined by
  $\theta(x)=[M_x,P^+_{x,2}]$.
	
  \begin{claim}\label{cl: theta is cont, QI}
    $\theta$ is continuous.
  \end{claim}
	
  With the notation of Section~\ref{ss: Gromov sp}, given
  $x=(x_n)\in\prod_{n=2}^\infty\{1,\dots,n\}$ and $R,r>0$, we have to
  prove that $\theta^{-1}(U_{R,r}(\theta(x)))$ is a neighborhood of
  $x$ in $\prod_{n=2}^\infty\{1,\dots,n\}$. Take some integer
  $n_0\ge2$ such that $e^2+\sum_{i=2}^{n_0}e^{i^2}+e^{n_0}>R$, and
  therefore
  $B_{M_x}(P^+_{x,2},R)\subset\bigcup_{n=2}^{n_0}M_{x,n}$. Let
  $\NN(x,n_0)$ be the open neighborhood of $x$ in
  $\prod_{n=2}^\infty\{1,\dots,n\}$ consisting of the elements
  $y=(y_n)$ such that $y_n=x_n$ if $n\le n_0$. Then
  $P^\pm_{x,n}=P^\pm_{y,n}$ for $2\le n\le n_0$ and $y\in V$,
  obtaining $d(P^+_{x,2},P^+_{y,2})=0$ and
  $H_{d,R}(M_x,P^+_{x,2};M_y,P^+_{y,2})=0$ for the isometric inclusion
  of $M_x$ and $M_y$ in $\R^2$ with $d$. Thus
  $\theta(\NN(x,n_0))\subset U_{R,r}(\theta(x))$, completing the proof
  of Claim~\ref{cl: theta is cont, QI}.
	
  \begin{claim}\label{cl: (theta times theta)(E_K_sigma) subset E_Lip}
    $(\theta\times\theta)(E_{K_\sigma})\subset E_{\text{\rm Lip}}$,
    and therefore $(\theta\times\theta)(E_{K_\sigma})\subset E_{QI}$.
  \end{claim}
	
  This claim can be easily proved as follows. Let $(x,y)\in
  E_{K_\sigma}$ for $x=(x_n)$ and $y=(y_n)$ in
  $\prod_{n=2}^\infty\{1,\dots,n\}$. Thus there is some $C\ge0$ such
  that $|x_n-y_n|\le C$ for all $n$. Consider the pointed bijection
  $\phi:(M_x,P^+_{x,2})\to(M_y,P^+_{y,2})$ defined by
  $\theta(P^\pm_{x,n})=P^\pm_{y,n}$. Then, with $\lambda=e^C$, we have
  \begin{multline*}
    d_x(P^+_{x,n},P^-_{x,n})=2e^{x_n}\le2e^{y_n+C}=\lambda\,d_y(P^+_{y,n},P^-_{y,n})\\
    =\lambda\,d_y(\phi(P^+_{x,n}),\phi(P^-_{x,n}))\;,
  \end{multline*}
  and, similarly,
  \[
  d_x(P^+_{x,n},P^-_{x,n})\ge\frac{1}{\lambda}\,d_y(\phi(P^+_{x,n}),\phi(P^-_{x,n}))\;.
  \]
  On the other hand, for $P\in M_{x,m}$ and $Q\in M_{x,n}$ with $m<n$,
  \begin{multline*}
    d_x(P,Q)=e^{x_m}+\sum_{i=m+1}^ne^{i^2}+e^{x_n}
    \le e^{y_m+C}+\sum_{i=m+1}^ne^{i^2}+e^{y_n+C}\\
    \le\lambda\left(e^{y_m}+\sum_{i=m+1}^ne^{i^2}+e^{y_n}\right)
    =\lambda\,d_y(\phi(P),\phi(Q))\;,
  \end{multline*}
  and, similarly,
  \[
  d_x(P,Q)\ge\frac{1}{\lambda}\,d_y(\phi(P),\phi(Q))\;.
  \]
  Thus $\phi$ is a $\lambda$-bi-Lipschitz bijection, completing the
  proof of Claim~\ref{cl: (theta times theta)(E_K_sigma) subset
    E_Lip}.

  \begin{claim}\label{cl: (theta times theta)^-1(E_QI) subset
      E_K_sigma}
    $(\theta\times\theta)^{-1}(E_{QI})\subset E_{K_\sigma}$.
  \end{claim}
	
  To prove this assertion, take some $x=(x_n)$ and $y=(y_n)$ in
  $\prod_{n=2}^\infty\{1,\dots,n\}$ such that
  $(\theta(x),\theta(y))\in E_{QI}$. Then, for some $C\ge0$ and
  $\lambda\ge1$, there are $C$-nets, $A\subset M$ and $B\subset M(y)$
  with $P^+_{x,2}\in A$ and $P^+_{y,2}\in B$, and there is a pointed
  $\lambda$-bi-Lipschitz bijection
  $\phi:(A,P^+_{x,2})\to(B,P^+_{y,2})$.
	
  \begin{claim}\label{cl: frac 1 n e^2n+1 > lambda}
    If $e^{n^2}\ge C$, $\frac{1}{n}e^{2n+1}>\lambda$ and
    $e^{(n+2)^2-(n+1)^2}>3\lambda$, then $\phi(M_{x,n}\cap A)\subset
    M_{y,n}$.
  \end{claim}
	
  Assume the conditions of this claim. Then $A\cap
  M_{x,m}\ne\emptyset$ for all $m\ge n$ by~\eqref{e^n^2 ge
    C}. Furthermore, for $2\le k<\ell\le n$,
  \begin{multline*}
    d_y(\phi(M_{x,n}\cap A),\phi(M_{x,n+1}\cap A))
    \ge\frac{1}{\lambda}\,d_x(M_{x,n}\cap A,M_{x,n+1}\cap A)\\
    >\frac{1}{\lambda}\,e^{(n+1)^2}>ne^{(n+1)^2-2n-1}=ne^{n^2}\ge2e^n+\sum_{i=3}^ne^{i^2}\\
    \ge e^{y_k}+\sum_{i=k+1}^\ell e^{i^2}+e^{y_\ell}=d_y(P',Q')
  \end{multline*}
  for all $P'\in M_{y,k}$ and $Q'\in M_{y,\ell}$. On the other hand,
  for $2\le k<\ell$ with $\ell\ge n+2$,
  \begin{multline*}
			d_y(\phi(P),\phi(Q))
			\le\lambda\,d_x(M_{x,n}\cap A,M_{x,n+1}\cap A)\\
			<\lambda(e^{(n+1)^2}+2e^{n+1})<\lambda3e^{(n+1)^2}<e^{(n+2)^2}\le e^{\ell^2}\\
			<e^{y_k}+\sum_{i=k+1}^\ell e^{i^2}+e^{y_\ell}=d_y(M_{y,k},M_{y,\ell})\;.
		\end{multline*}
	for all $P\in M_{x,n}\cap A$ and $Q\in M_{x,n+1}\cap A$. Therefore, either
		\begin{equation}\label{phi(M_x,n cap A subset M_y,n & phi(M_x,n+1 cap A subset M_y,n+1}
			\phi(M_{x,n}\cap A)\subset M_{y,n}\quad\&\quad\phi(M_{x,n+1}\cap A)\subset M_{y,n+1}\;,
		\end{equation}
	or
		\begin{equation}\label{phi((M_x,n cup M_x,n+1) cap A) subset M_y,m}
			\phi((M_{x,n}\cup M_{x,n+1})\cap A)\subset M_{y,m}
		\end{equation}
	for some $m$. In the case~\eqref{phi((M_x,n cup M_x,n+1) cap A) subset M_y,m}, we have
		\begin{multline*}
			2e^m=d_y(\phi(M_{x,n}\cap A),\phi(M_{x,n+1}\cap A))\\
			\ge\frac{1}{\lambda}\,d_x(M_{x,n}\cap A,M_{x,n+1}\cap A)>e^{(n+1)^2}/\lambda\;,
		\end{multline*}
	giving $m>(n+1)^2-\ln(2\lambda)$. Applying this to $n+1$ and $n+2$, we get that, either
		\begin{equation}\label{phi(M_x,n+1 cap A subset M_y,n+1 & phi(M_x,n+2 cap A subset M_y,n+2}
			\phi(M_{x,n+1}\cap A)\subset M_{y,n+1}\quad\&\quad\phi(M_{x,n+2}\cap A)\subset M_{y,n+2}\;,
		\end{equation}
	or
		\begin{equation}\label{phi((M_x,n cup M_x,n+1) cap A) subset M_y,m'}
			\phi((M_{x,n+1}\cup M_{x,n+2})\cap A)\subset M_{y,m'}
		\end{equation}
	for some $m'>(n+2)^2-\ln(2\lambda)$. If~\eqref{phi((M_x,n cup M_x,n+1) cap A) subset M_y,m} and~\eqref{phi((M_x,n cup M_x,n+1) cap A) subset M_y,m'} hold, then $m=m'$ and
        \[
        \phi((M_{x,n}\cup M_{x,n+1}\cup M_{x,n+2})\cap A)\subset
        M_{y,m}\;,
        \]
	which is a contradiction because $\phi$ is a bijection whereas
        \[
        \#((M_{x,n}\cup M_{x,n+1}\cup M_{x,n+2})\cap A)\ge3>2=\#M_{y,m}\;.
        \]
        If~\eqref{phi((M_x,n cup M_x,n+1) cap A) subset M_y,m}
        and~\eqref{phi(M_x,n+1 cap A subset M_y,n+1 & phi(M_x,n+2 cap
          A subset M_y,n+2} hold, then $n+1=m>(n+1)^2-\ln(2\lambda)$,
        which contradicts the condition
        $e^{(n+2)^2-(n+1)^2}>3\lambda$.  So~\eqref{phi(M_x,n cap A
          subset M_y,n & phi(M_x,n+1 cap A subset M_y,n+1} must be
        true, showing Claim~\ref{cl: frac 1 n e^2n+1 > lambda}.
	
	From Claim~\ref{cl: frac 1 n e^2n+1 > lambda}, it easily follows that
		\begin{equation}\label{phi(M_x,n cap A) = M_y,n cap B}
			\phi(M_{x,n}\cap A)=M_{y,n}\cap B
		\end{equation}
	for $n$ large enough. Suppose first that $M_{x,n}\subset A$ for such an $n$, and therefore $M_{y,n}\subset B$ by~\eqref{phi(M_x,n cap A) = M_y,n cap B}. Thus
		\[
			2e^{y_n}=d_y(P^+_{y,n},P^-_{y,n})=d_y(\phi(P^+_{x,n}),\phi(P^-_{x,n}))
			\ge\frac{1}{\lambda}\,d_x(P^+_{x,n},P^-_{x,n})=\frac{2e^{x_n}}{\lambda}\;,
		\]
	giving $y_n\ge x_n-\ln\lambda$. Similarly, $y_n\le x_n+\ln\lambda$, obtaining $|x_n-y_n|\le\ln\lambda$. 
	
	Now, assume that $M_{x,n}\not\subset A$ for such an $n$; in particular, $C>0$. Then $M_{y,n}\not\subset B$ by~\eqref{phi(M_x,n cap A) = M_y,n cap B}. So $e^{x_n},e^{y_n}\le C/2$ by~\eqref{e^n^2 ge C & e^{x_n}>C/2}, giving $x_n,y_n\le\ln(C/2)$, and therefore $|x_n-y_n|\le\ln(C/2)$.
	
	Hence $|x_n-y_n|\le\max\{\ln\lambda,\ln(C/2)\}$ for all $n$ large enough, and therefore $\sup_n|x_n-y_n|<\infty$, obtaining that $(x,y)\in E_{K_\sigma}$. This completes the proof of Claim~\ref{cl: (theta times theta)^-1(E_QI) subset E_K_sigma}.
	
	Claims~\ref{cl: theta is cont, QI},~\ref{cl: (theta times theta)(E_K_sigma) subset E_Lip} and~\ref{cl: (theta times theta)^-1(E_QI) subset E_K_sigma} show that $\theta$ realizes the reduction $E_{K_\sigma}\le_cE_{QI}$.
	
	A similar argument with a slight modification of the definition of $M(x)$, using $P^\pm_{x,n}=(\sum_{i=2}^ne^{i^2},\pm x_n)$, shows that $E_{K_\sigma}\le_BE_{GH}$.
\end{proof}

\begin{rem}
  In Claim~\ref{cl: theta is cont, QI}, $\theta$ is in fact a
  topological embedding, as shows the following argument. First, let
  us prove that $\theta$ is injective. Suppose that
  $\theta(x)=\theta(y)$ for some $x=(x_n)$ and $y=(y_n)$ in
  $\prod_{n=2}^\infty\{1,\dots,n\}$. This means that there is a
  pointed isometry $\phi:(M_xP^+_{x,2})\to(M_y,P^+_{y,2})$. We get
  $\phi(M_{x,n})=M_{y,n}$ for all $n\ge2$ by Claim~\ref{cl: frac 1 n
    e^2n+1 > lambda} with $A=M_x$, $B=M_y$, $C=0$ and $\lambda=1$; in
  fact, the argument can be simplified in this case. Hence, for each
  $n\ge2$,
  \[
  2e^{x_n}=d_x(P^+_{x,n},P^-_{x,n})=d_y(\phi(P^+_{x,n}),\phi(P^-_{x,n}))=d_y(P^+_{y,n},P^-_{y,n})=2e^{y_n}\;,
  \]
  giving $x_n=y_n$. Thus $x=y$.
	
  Finally, let us prove that
  $\phi^{-1}:\phi(\prod_{n=2}^\infty\{1,\dots,n\})\to\prod_{n=2}^\infty\{1,\dots,n\}$
  is continuous at $\phi(x)$ for every
  $x=(x_n)\in\prod_{n=2}^\infty\{1,\dots,n\}$. With the notation of
  the proof of Claim~\ref{cl: theta is cont, QI}, we have to check
  that, for all $n_0\ge2$, there is some $R,r>0$ so that
  $\phi^{-1}(U_{R,r}(\theta(x)))\subset\NN(x,n_0)$. Let
  $y=(y_n)\in\prod_{n=2}^\infty\{1,\dots,n\}$ such that $\theta(y)\in
  U_{R,r}(\theta(x))$ for some $R,r>0$ to be determined later. Then
  there is a metric $d'$ on $M_x\sqcup M_y$, extending $d_x$ and
  $d_y$, such that $d'(P^+_{x,2},P^+_{y,2})<r$ and
  $H_{d',R}(M_x,P^+_{x,2};M_y,P^+_{y,2})<r$. Since $e^n<e^{(n+1)^2}$
  for all $n\ge2$, we can take $R$ such that
  \[
  e^2+\sum_{i=2}^{n_0}e^{i^2}+e^{n_0}<R<e^2+\sum_{i=2}^{n_0+1}e^{i^2}\;,
  \]
  and therefore $B_{M_x}(P^+_{x,2},R)=\bigcup_{n=2}^{n_0}M_{x,n}$ and
  $B_{M_y}(P^+_{y,2},R)=\bigcup_{n=2}^{n_0}M_{y,n}$. So, for each
  $P^\pm_{x,n}$ with $2\le n\le n_0$, there is some $\widehat
  P^\pm_{x,n}\in M_y$ such that $d(P^\pm_{x,n},\widehat
  P^\pm_{x,n})<r$; in particular, we can take $\widehat
  P^\pm_{x,2}=P^\pm_{y,2}$. Let $\widehat M_{x,n}=\{\widehat
  P^+_{x,n},\widehat P^-_{x,n}\}$ for $2\le n\le n_0$.  Choose $r$
  such that $r<1$ and $e^n+r<e^{(n+1)^2}$ for $2\le n\le n_0$. So
  $\widehat M_{x,n}=M_{y,n}$ for $2\le n\le n_0$. Then, by the
  triangle inequality,
  \begin{multline*}
    2e^{x_n}=d_x(P^+_{x,n},P^-_{x,n})\le d_y(\widehat P^+_{x,n},\widehat P^-_{x,n})+2r\\
    =d_y(P^+_{y,n},P^-_{y,n})+2r=2e^{y_n}+2r\;,
  \end{multline*}
  giving $e^{x_n}\le e^{y_n}+r$. Similarly, we get $e^{x_n}\ge
  e^{y_n}-r$. Thus $|e^{x_n}-e^{y_n}|\le r$, obtaining $x_n=y_n$
  because $r<1$. Therefore $y\in\NN(x,n_0)$, as desired.
\end{rem}

\begin{rem}
  According to Claim~\ref{cl: (theta times theta)(E_K_sigma) subset
    E_Lip}, the map $\theta$ of the proof of Proposition~\ref{p: E_GH
    and E_QI} also gives the reduction $E_{K_\sigma}\le_cE_{\text{\rm
      Lip}}$. An analogous property is satisfied with another point of
  view: considering Polish metric spaces as the elements of the space
  of closed subspaces of some universal Polish metric space, like the
  Urysohn space, the relation given by the existence of bi-Lipschitz
  bijections is Borel bi-reducible with $E_{K_\sigma}$
  \cite[Theorem~24]{Rosendal2005}.
\end{rem}



\providecommand{\bysame}{\leavevmode\hbox to3em{\hrulefill}\thinspace}
\providecommand{\MR}{\relax\ifhmode\unskip\space\fi MR }
\providecommand{\MRhref}[2]{%
  \href{http://www.ams.org/mathscinet-getitem?mr=#1}{#2}
}
\providecommand{\href}[2]{#2}

\end{document}